\documentclass[11pt]{amsart}

\usepackage[utf8]{inputenc}
\usepackage[T1]{fontenc}
\usepackage[english]{babel}
\usepackage{amsmath,amssymb,amsthm,mathtools}
\usepackage{geometry}
\usepackage{hyperref}

\hypersetup{
    colorlinks=true,
    linkcolor=blue,
    citecolor=blue,
    urlcolor=blue
}

\newtheorem{theorem}{Theorem}
\newtheorem{proposition}{Proposition}
\newtheorem{lemma}{Lemma}
\newtheorem{corollary}{Corollary}

\newcommand{\R}{\mathbb R}
\newcommand{\E}{\mathbb E}
\newcommand{\Prob}{\mathbb P}
\newcommand{\cP}{\mathcal P}
\newcommand{\eps}{\varepsilon}
\newcommand{\norm}[1]{\left\|#1\right\|}
\newcommand{\Var}{\operatorname{Var}}
\newcommand{\gauss}{\gamma_{m,n}}
\newcommand{\N}{\mathbb N}
\newcommand{\Nm}{\mathbb N_0}

\title[Almost norming vertices]{Almost norming vertices for homogeneous polynomials on the cube}

\author{Dami\'an Pinasco}
\address{Universidad Torcuato Di Tella.
Departamento de Matem\'atica y Estad\'istica.
CONICET.
Av. Figueroa Alcorta 7350,
C1428BCW Ciudad de Buenos Aires,
Argentina}
\email{dpinasco@utdt.edu}

\author{Ignacio Zalduendo}
\address{Universidad Torcuato Di Tella.
Departamento de Matem\'atica y Estad\'istica.
Av. Figueroa Alcorta 7350,
C1428BCW Ciudad de Buenos Aires,
Argentina}
\email{izalduendo@utdt.edu}



\date{}

\begin{document}

\begin{abstract}
Let \(m\geq1\) be fixed. We consider the space \(\cP_m(\R^n)\) of real
\(m\)-homogeneous polynomials on \(\R^n\), endowed with the standard Gaussian
measure \(\gamma_{m,n}\) associated with the Bombieri norm. We study how well
the norm of a typical polynomial on the unit ball of \(\ell_\infty^n\), namely
the cube \([-1,1]^n\), can be recovered from its values at the vertices.

For \(P\in\cP_m(\R^n)\), set
\[
M(P)=\max_{x\in[-1,1]^n}|P(x)|,
\qquad
V(P)=\max_{\eps\in\{-1,1\}^n}|P(\eps)|.
\]
If \(P_n\) is chosen according to \(\gamma_{m,n}\), we prove that the relative
loss
\[
1-\frac{V(P_n)}{M(P_n)}
\]
is of order at most \(n^{-1/2}\) in probability.
Consequently, for every \(0<\beta<1/2\),
\[
\gamma_{m,n}\left\{
P\in\cP_m(\R^n):
V(P)\geq (1-n^{-\beta})M(P)
\right\}
\longrightarrow1
\]
as \(n\to\infty\). Thus, with respect to the Bombieri Gaussian measure, the
vertices of the cube are asymptotically norming for homogeneous polynomials of
fixed degree.
\end{abstract}

\maketitle

\section{Introduction}

The problem considered in this paper belongs to the general question of where homogeneous polynomials on
finite-dimensional normed spaces attain their norm. More precisely, if \(A\) is
a prescribed subset of the unit sphere, one may ask how large is the set of
norm-one polynomials whose norm is attained on \(A\). This point of view was
adopted by Carando and Zalduendo in \cite{CarandoZalduendo2003}, with special
attention to the case of polynomials on \(\ell_\infty^n\).

The exact vertex problem asks whether, for large dimension, most homogeneous
polynomials attain their norm at a vertex. Once a probability measure $\gauss$ on
the space of polynomials has been chosen, this becomes the question whether
\[
\gauss\{P : P \text{ attains its norm at a vertex }\}\longrightarrow1.
\]

The geometry of the underlying unit ball is relevant. The corresponding problem
for \(\ell_1^n\) was studied by P\'erez-Garc\'ia and Villanueva
\cite{PerezGarciaVillanueva2004}. They prove that in the quadratic case ($m=2$) the result is false.
However, it is shown in \cite{PinascoZalduendo2012} that degree $m \geq 3$
exhibits a different behaviour. This is one of the reasons why the vertex
problem for the cube should be treated as a specific question rather than as a
formal consequence of a general finite-dimensional principle.

In order to turn this question into a quantitative one, a measure on
the space of homogeneous polynomials must be chosen. This is not merely a formal issue: this space
does not carry a preferred Euclidean structure coming
from the normed-space geometry of \(\ell_\infty^n\). Following the construction
used in \cite{PinascoZalduendo2012}, we regard the space of $m$-homogeneous polynomials over $\R^n$,
which we denote by
\(\cP_m(\R^n)\), as the dual of the
symmetric tensor product \(\otimes^{m,s}\R^n\), endowed with the Hilbert space
structure induced by the Euclidean structure of \(\R^n\). The resulting norm on
\(\cP_m(\R^n)\) is the Bombieri norm, a Hilbertian norm on polynomial spaces
appearing naturally in the study of polynomials in several variables; see, for
instance, \cite{BeauzamyBombieriEnfloMontgomery1990}. We shall work with the
standard Gaussian measure corresponding to this Hilbert space structure, and
denote it by \(\gauss\). The precise coordinate expression of this measure is
recalled in the next section.

Several related works study local versions of the vertex problem. In
\cite{PinascoZalduendo2019}, the probability that a homogeneous polynomial has a
local maximum at a vertex is computed in terms of the sharpness of the vertex.
In \cite{PinascoSmuclerZalduendo2021}, local maxima at vertices of the simplex
are studied through orthant probabilities for correlated normal random variables.
These approaches use first-order conditions: at a vertex one studies the signs
of suitable directional derivatives. The problem is translated into the
study of Gaussian orthant probabilities.

We work on \(\R^n\) endowed with the sup norm. Its unit ball is the cube
\[
B_{\ell_\infty^n}=[-1,1]^n,
\]
and its extreme points are the vertices
\[
\{-1,1\}^n.
\]

For \(P\in\cP_m(\R^n)\), we set
\[
M(P)=\max \{|P(x)| : x \in [-1,1]^n \}
\]
and
\[
V(P)=\max \{|P(\varepsilon)| : \varepsilon \in \{-1,1\}^n \}.
\]
Thus \(M(P)\) is the norm of \(P\) on the cube, while \(V(P)\) is the largest
value attained at the vertices. Clearly, $0\leq V(P)\leq M(P)$.

The equality
\[
M(P)=V(P)
\]
is the exact vertex-attainment property. A quantitative way to approach this
problem is to measure how close the vertices are to being norming. We consider
the relative defect
\[
1-\frac{V(P)}{M(P)}.
\]
Thus, \(V(P)\geq(1-\eta)M(P)\) precisely when the vertices are \(\eta\)-norming
for \(P\). Our main quantitative estimate shows that, if \(P_n\) is chosen
according to \(\gamma_{m,n}\), then
\[
1-\frac{V(P_n)}{M(P_n)}
=
O_{\Prob}(n^{-1/2}).
\]
In particular, for every \(0<\beta<1/2\),
\[
n^\beta\left(
1-\frac{V(P_n)}{M(P_n)}
\right)
\longrightarrow0
\]
in probability. Consequently,
\[
\gamma_{m,n}\left\{
P\in\cP_m(\R^n):
V(P)\geq (1-n^{-\beta})M(P)
\right\}
\longrightarrow1.
\]

The proof is based on a structural decomposition of polynomials into a sum of their
multilinear component and repeated-variable component. Since the multilinear component attains its norm at vertices and
asymptotically dominates the repeated-variable component in supremum norm, our
result follows from quantitative comparison estimates.

For a polynomial
\(P_n\) chosen according to \(\gauss\), we write
\[
P_n=Q_n+R_n,
\]
where \(Q_n\) is the multilinear part of \(P_n\), and \(R_n\) collects all
monomials with repeated variables. 
Multilinear polynomials attain their sup-norm on the cube at vertices. The main point of the proof is that, when the
degree is fixed and the dimension tends to infinity, the multilinear part is the dominant contribution to the norm on the cube in
high dimension. More precisely, the norm of
\(Q_n\) grows like \(n^{(m+1)/2}\), while the repeated-variable part $R_n$ grows at
most like \(n^{m/2}\), with high probability.

This gap in the orders of growth implies that the contribution of \(R_n\) is
negligible in comparison with that of \(Q_n\). Since \(Q_n\) is normed by the
vertices, a deterministic comparison then shows that the vertices are almost norming for \(P_n\). 
This yields the asserted relative error estimate.

The appearance of the exponent \(1/2\) is a consequence of the combinatorial
gap between square-free and repeated-variable monomials. The multilinear part contains \(\binom{n}{m}\) square-free monomials,
whereas the largest repeated-variable blocks contain only
\(\binom{n}{m-1}\) monomials. This
difference in size ultimately produces the factor \(n^{-1/2}\) governing the
relative error.

Thus the exact question
\[
\gauss\{P:M(P)=V(P)\}\longrightarrow1
\]
is not decided by this estimate. The result proved here gives a quantitative asymptotic form
of the vertex problem: although the maximum need not be attained exactly at a
vertex, any possible gain obtained at non-vertex points is negligible in
relative norm, with defect bounded by \(O_{\Prob}(n^{-1/2})\).

The proof combines elementary estimates, Gaussian process methods, and Gaussian
concentration for Lipschitz functions.

The paper is organized as follows. In Section 2 we collect the preliminary material on homogeneous polynomials, the
Bombieri norm, the corresponding Gaussian measure, and the probabilistic
notation used throughout the paper. In Section 3 we give a deterministic reduction showing that it is
enough to control the repeated-variable part of the polynomial. In Section 4 we
estimate the multilinear part. In Section 5 we control the repeated-variable
part. In Section 6 we prove the quantitative defect estimate and derive the almost
norming result as a corollary.

\section{Preliminaries}

We shall use the following standard probabilistic notation. The expression
\[
A_n=O_{\Prob}(a_n)
\]
means that \(A_n\) is bounded above by a constant multiple of \(a_n\), up to an
event whose probability can be made arbitrarily small. Equivalently, if
\((A_n)\) is a sequence of nonnegative random variables and \((a_n)\) is a
sequence of positive numbers, then \(A_n=O_{\Prob}(a_n)\) means that for every
\(\eta>0\) there exists \(C>0\) such that
\[
\Prob(A_n>C a_n)<\eta
\]
for all sufficiently large \(n\). 

We denote by
\[
\cP_m(\R^n)
\]
the real vector space of all \(m\)-homogeneous polynomials on \(\R^n\), that is,
all polynomials \(P:\R^n\to\R\) satisfying
\[
P(\lambda x)=\lambda^mP(x)
\qquad
(\lambda\in\R,\ x\in\R^n).
\]
Equivalently, every \(P\in\cP_m(\R^n)\) can be written as
\[
P(x)=
\sum_{\substack{\alpha\in\Nm^n\\ |\alpha|=m}}
a_\alpha x^\alpha,
\]
where
\[
|\alpha|=\alpha_1+\cdots+\alpha_n,
\qquad
x^\alpha=x_1^{\alpha_1}\cdots x_n^{\alpha_n}.
\]
Throughout the paper the degree \(m\) is fixed, whereas the dimension \(n\)
tends to infinity.

The Bombieri norm on \(\cP_m(\R^n)\) is given by
\[
\norm{P}_B^2=
\sum_{\substack{\alpha\in\Nm^n\\ |\alpha|=m}}
a_\alpha^2\frac{\alpha!}{m!},
\qquad
\alpha!=\alpha_1!\cdots\alpha_n!.
\]
This is the norm induced by the Hilbert structure on the symmetric tensor
product. We refer to \cite{PinascoZalduendo2012} for its use in this
probabilistic context, and to \cite{BeauzamyBombieriEnfloMontgomery1990} for
the Bombieri norm in polynomial inequalities.

Let \(\gauss\) denote the standard Gaussian measure associated with the Hilbert
space structure on \(\cP_m(\R^n)\) whose norm is \(\norm{\cdot}_B\). In
Bombieri coordinates, a polynomial chosen according to \(\gauss\) can be written
as
\[
P_n(x)=
\sum_{\substack{\alpha\in\Nm^n\\ |\alpha|=m}}
\sqrt{\binom{m}{\alpha}}\,g_\alpha x^\alpha,
\]
where \((g_\alpha)_{|\alpha|=m}\) are independent standard real Gaussian
variables and
\[
\binom{m}{\alpha}=\frac{m!}{\alpha_1!\cdots\alpha_n!}.
\]
This is the standard Gaussian measure associated with the Bombieri norm. In the
literature on random polynomials, the same normalization is often referred to as
the Bombieri--Weyl or Kostlan Gaussian; see, for instance,
\cite{Kostlan1993} and \cite{EdelmanKostlan1995}.

A direct computation gives, for all \(x,y\in\R^n\),
\[
\E[P_n(x)P_n(y)]
=
\sum_{\substack{\alpha\in\Nm^n\\ |\alpha|=m}}
\binom{m}{\alpha}x^\alpha y^\alpha
=
\left(\sum_{j=1}^n x_jy_j\right)^m
=
\langle x,y\rangle^m.
\]
In particular, if \(U\) is an orthogonal transformation, then
\[
P_n\circ U\stackrel{d}{=}P_n.
\]
Thus the chosen probability does not favor any Euclidean direction. This is the
sense in which the measure is compatible with the spherical geometry of
\(\R^n\).

Before entering the proof, we dispose of the linear case. If \(m=1\), then
\(P_n\) is a linear functional, and therefore its maximum and minimum on the
cube are attained at vertices. Hence
\[
M(P_n)=V(P_n)
\]
identically. Thus, from now on, we assume \(m\geq2\).

\section{A deterministic reduction}

For a polynomial chosen according to \(\gauss\), and for \(n\geq m\), we split
\[
P_n=Q_n+R_n,
\]
where \(Q_n\) is the multilinear part and \(R_n\) contains all monomials in
which at least one variable appears with exponent at least \(2\). Since the
multilinear part corresponds exactly to the square-free multiindices, and each
of them is uniquely represented by a set
\[
\{i_1<\cdots<i_m\}\subset\{1,\dots,n\},
\]
we may write
\[
Q_n(x)=
\sqrt{m!}
\sum_{1\leq i_1<\cdots<i_m\leq n}
g_{i_1,\dots,i_m}x_{i_1}\cdots x_{i_m}.
\]

We shall use the following elementary fact.

\begin{lemma} 
If $Q$ is a homogeneous multilinear polynomial, then $M(Q)=V(Q)$.
\end{lemma}

\begin{proof}
Write
\[
Q(x)=\sum_{|S|=m}a_S\prod_{i\in S}x_i.
\]
Then \(Q\) is affine in each coordinate separately. Fix all coordinates
except \(x_j\). If \(x_j\in[-1,1]\), write
\[
x_j=t\cdot 1+(1-t)(-1)
\]
for some \(t\in[0,1]\). By affinity in the \(j\)-th coordinate,
\[
Q(x_1,\dots,x_j,\dots,x_n)
=
tQ(x_1,\dots,1,\dots,x_n)
+
(1-t)Q(x_1,\dots,-1,\dots,x_n).
\]
Therefore the value of \(Q\) at the original point lies between the two values
obtained by replacing \(x_j\) by \(1\) and by \(-1\). Hence one of these two
endpoint replacements does not decrease the value of \(Q\), and one of them
does not increase it.

Iterating this argument over the coordinates, the maximum of \(Q\) on the cube
is attained at a vertex, and the minimum of \(Q\) on the cube is also attained at
a vertex. Since
\[
M(Q)
=
\max\left\{
\max_{x\in[-1,1]^n}Q(x),
-\min_{x\in[-1,1]^n}Q(x)
\right\},
\]
the supremum norm of \(Q\) on the cube is attained at a vertex.
\end{proof}

For \(n\geq m\), square-free multiindices of degree \(m\) exist. Hence the
condition \(Q_n\equiv0\) is equivalent to the vanishing of all square-free
coefficients, and therefore defines a proper linear subspace of
\(\cP_m(\R^n)\). Since \(\gauss\) is a nondegenerate Gaussian measure on
\(\cP_m(\R^n)\), this subspace has \(\gauss\)-measure zero.

\begin{proposition}
For \(n\geq m\), outside a \(\gauss\)-null set, define
\[
t_n:=\frac{M(R_n)}{M(Q_n)}.
\]
Then
\[
1-\frac{V(P_n)}{M(P_n)}\leq 2t_n.
\]
In particular, if \(t_n\to0\) in probability, then
\[
\frac{V(P_n)}{M(P_n)}\to1
\]
in probability.
\end{proposition}

\begin{proof}
By the triangle inequality,
\[
M(P_n)
\leq
M(Q_n)
+
M(R_n).
\]
On the vertices, the reverse triangle inequality gives
\[
V(P_n)
=
V(Q_n + R_n)
\geq
V(Q_n)
-
V(R_n).
\]
By the previous lemma, $V(Q_n)=M(Q_n)$, 
and since $V(R_n) \leq M(R_n)$, 
\[
V(P_n)
\geq
M(Q_n)
-
M(R_n).
\]
Combining the upper estimate for \(M(P_n)\) with the lower
estimate for \(V(P_n)\) gives
\[
\frac{V(P_n)}{M(P_n)}
\geq
\frac{1-t_n}{1+t_n}.
\]
Hence
\[
1-\frac{V(P_n)}{M(P_n)}
\leq
1-\frac{1-t_n}{1+t_n}
=
\frac{2t_n}{1+t_n}
\leq
2t_n.
\]
\end{proof}

\section{The size of the multilinear part}

In this section we estimate the size of the multilinear part \(Q_n\) on the
cube. The upper estimate is a direct consequence of a Gaussian tail bound and a
union bound over the vertices. The lower estimate is less immediate: we first
construct exponentially many vertices which are well separated in Hamming
distance, and then use Sudakov's minoration for the corresponding Gaussian
process.

We shall use the elementary Gaussian tail estimate
\[
\Prob(|Z|>t)\leq 2\exp\left(-\frac{t^2}{2\sigma^2}\right),
\]
valid for every centered Gaussian random variable \(Z\) with variance
\(\sigma^2\). We begin with the upper estimate.

\begin{proposition}
For fixed \(m\),
\[
M(Q_n)=O_{\Prob}(n^{(m+1)/2}).
\]
More precisely, if \(C>\sqrt{2\log2}\), then
\[
\Prob\left(M(Q_n)>C n^{(m+1)/2}\right)\to0.
\]
\end{proposition}

\begin{proof}
By the vertex lemma, $M(Q_n)=V(Q_n)$.
For fixed \(\eps\), the random variable
\[
Q_n(\eps)=
\sqrt{m!}\sum_{1\leq i_1<\cdots<i_m\leq n}
 g_{i_1,\dots,i_m}\eps_{i_1}\cdots\eps_{i_m}
\]
is centered Gaussian with variance
\[
m!\binom{n}{m}.
\]
Since
\[
m!\binom{n}{m}
=
n(n-1)\cdots(n-m+1)
\leq n^m,
\]
the Gaussian tail estimate gives
\[
\Prob(|Q_n(\eps)|>t)
\leq
2\exp\left(-\frac{t^2}{2n^m}\right).
\]
By the union bound over the \(2^n\) vertices,
\[
\Prob(M(Q_n)>t)
\leq
2^{n+1}\exp\left(-\frac{t^2}{2n^m}\right).
\]
Taking \(t=Cn^{(m+1)/2}\), we obtain
\[
\Prob\left(M(Q_n)>Cn^{(m+1)/2}\right)
\leq
2\exp\left(n\log2-\frac{C^2}{2}n\right).
\]
This tends to zero whenever
\[
C>\sqrt{2\log2}.
\]
\end{proof}

We now turn to the lower estimate. In order to prove this estimate, we use three auxiliary ingredients: an entropy
bound, a construction of binary codes with two-sided Hamming distance bounds, and a canonical
separation estimate for the associated Gaussian process.

\begin{lemma}
If \(0<\delta<1/2\), then
\[
\sum_{k=0}^{\lfloor\delta n\rfloor}\binom{n}{k}
\leq
\exp(H(\delta)n),
\]
where
\[
H(\delta)=-\delta\log\delta-(1-\delta)\log(1-\delta).
\]
\end{lemma}

\begin{proof}
For \(0<t<1\),
\[
(1+t)^n=\sum_{k=0}^n\binom{n}{k}t^k.
\]
Since \(t^k\geq t^{\delta n}\) for \(k\leq\delta n\),
\[
\sum_{k=0}^{\lfloor\delta n\rfloor}\binom{n}{k}
\leq t^{-\delta n}(1+t)^n.
\]
Choosing \(t=\delta/(1-\delta)\) gives
\[
t^{-\delta n}(1+t)^n
=
\exp(H(\delta)n),
\]
and the result follows.
\end{proof}

For \(\eps,\eta\in\{-1,1\}^n\), let
\[
d_H(\eps,\eta)=\#\{j:\eps_j\neq\eta_j\}
\]
be their Hamming distance.

\begin{lemma}
There are constants \(\delta\in(0,1/2)\), \(c>0\), and \(n_0\in\N\) such that,
for every \(n\geq n_0\), there is a set
\[
\mathcal A_n\subset\{-1,1\}^n
\]
with
\[
|\mathcal A_n|\geq e^{cn}
\]
and such that every distinct \(\eps,\eta\in\mathcal A_n\) satisfy
\[
\delta n\leq d_H(\eps,\eta)\leq (1-\delta)n.
\]
\end{lemma}

\begin{proof}
Choose \(0<\delta<1/2\), and set
\[
\kappa_\delta=\log2-H(\delta)>0.
\]
Let
\[
N=\lfloor e^{cn}\rfloor,
\]
where \(c>0\) will be chosen below. Choose \(N\) independent random vertices
\[
\xi^1,\dots,\xi^N
\]
of \(\{-1,1\}^n\), each uniformly distributed on the set of all vertices.

For two independent uniform vertices \(\xi,\xi'\), the random variable
\(d_H(\xi,\xi')\) has distribution \(\operatorname{Bin}(n,1/2)\). Hence
\[
\Prob(d_H(\xi,\xi')\leq \delta n)
=
2^{-n}
\sum_{k=0}^{\lfloor \delta n\rfloor}
\binom{n}{k}.
\]
By the entropy bound,
\[
\Prob(d_H(\xi,\xi')\leq \delta n)
\leq
\exp(-\kappa_\delta n).
\]
By symmetry of the binomial distribution,
\[
\Prob(d_H(\xi,\xi')\geq (1-\delta)n)
\leq
\exp(-\kappa_\delta n).
\]
Therefore
\[
\Prob\left(
d_H(\xi,\xi')\notin[\delta n,(1-\delta)n]
\right)
\leq
2\exp(-\kappa_\delta n).
\]

By the union bound over the pairs \(1\leq p<q\leq N\),
\[
\Prob\left(
\exists\,p<q:
d_H(\xi^p,\xi^q)\notin[\delta n,(1-\delta)n]
\right)
\leq
N^2\,2\exp(-\kappa_\delta n).
\]
Since \(N\leq e^{cn}\), this is bounded by
\[
2\exp((2c-\kappa_\delta)n).
\]
Choose \(c>0\) such that
\[
2c<\kappa_\delta.
\]
Then the last quantity tends to zero. Hence, for all sufficiently large \(n\),
there exists a choice of vertices
\[
\xi^1,\dots,\xi^N
\]
such that every pair satisfies
\[
\delta n\leq d_H(\xi^p,\xi^q)\leq(1-\delta)n.
\]
Taking
\[
\mathcal A_n=\{\xi^1,\dots,\xi^N\}
\]
gives the desired set.
\end{proof}

We now view the values of \(Q_n\) at the vertices as a Gaussian process. For
each \(\eps\in\{-1,1\}^n\), set
\[
X_\eps=Q_n(\eps).
\]
Thus \((X_\eps)_{\eps\in\{-1,1\}^n}\) is a centered Gaussian process indexed by
the vertices of the cube.

Although \(Q_n\) is a homogeneous Gaussian chaos of order \(m\), we shall not
use the general theory of Gaussian chaoses. The estimates needed below follow
from the explicit covariance structure of the square-free process, elementary
metric arguments, Sudakov's minoration, and Gaussian concentration. For general
background on Gaussian processes and related concentration methods, we refer to
\cite{LedouxTalagrand,TalagrandUpperLower}; for Gaussian chaoses, see for
instance \cite{Latala2006}.

Its canonical metric is the \(L^2\)-distance between increments:
\[
d(\eps,\eta)
=
\left(\E[(X_\eps-X_\eta)^2]\right)^{1/2}.
\]
This is the metric which enters Sudakov's minoration.

\begin{lemma}
Let \(0<\delta<1/2\). If
\[
\delta n\leq d_H(\eps,\eta)\leq(1-\delta)n,
\]
then there is a constant \(b_{m,\delta}>0\) such that
\[
d(\eps,\eta)\geq b_{m,\delta}n^{m/2}
\]
for all sufficiently large \(n\).
\end{lemma}

\begin{proof}
For a set \(S\subset\{1,\dots,n\}\), write
\[
\eps_S=\prod_{i\in S}\eps_i.
\]
Since
\[
Q_n(\eps)
=
\sqrt{m!}\sum_{\substack{S\subset\{1,\dots,n\}\\ |S|=m}}
g_S\eps_S,
\]
we have
\[
X_\eps-X_\eta
=
\sqrt{m!}\sum_{\substack{S\subset\{1,\dots,n\}\\ |S|=m}}
g_S(\eps_S-\eta_S).
\]
The coefficients \(g_S\) are independent standard Gaussian variables. Therefore
\[
d(\eps,\eta)^2
=
\E[(X_\eps-X_\eta)^2]
=
m!\sum_{\substack{S\subset\{1,\dots,n\}\\ |S|=m}}
(\eps_S-\eta_S)^2.
\]

Let
\[
D=\{i:\eps_i\neq\eta_i\},
\qquad
h=|D|=d_H(\eps,\eta).
\]
We count only those sets \(S\) which contain exactly one element of \(D\). For
such an \(S\), the products \(\eps_S\) and \(\eta_S\) differ by exactly one sign
change, and hence
\[
\eps_S=-\eta_S.
\]
Therefore
\[
(\eps_S-\eta_S)^2=4.
\]
The number of sets \(S\) of cardinality \(m\) with exactly one element in \(D\)
is
\[
h\binom{n-h}{m-1}.
\]
Under the assumption
\[
\delta n\leq h\leq (1-\delta)n,
\]
we have
\[
h\geq \delta n
\qquad\text{and}\qquad
n-h\geq \delta n.
\]
Since \(m\) is fixed, there is a constant \(c_{m,\delta}>0\) such that
\[
h\binom{n-h}{m-1}\geq c_{m,\delta}n^m
\]
for all sufficiently large \(n\). Hence
\[
d(\eps,\eta)^2
\geq
4m!c_{m,\delta}n^m.
\]
Taking square roots gives
\[
d(\eps,\eta)\geq b_{m,\delta}n^{m/2},
\]
for a suitable constant \(b_{m,\delta}>0\).
\end{proof}

\begin{lemma}[Sudakov minoration]
Let \((X_t)_{t\in T}\) be a centered Gaussian process with canonical metric
\[
d(s,t)=\left(\E[(X_s-X_t)^2]\right)^{1/2}.
\]
If \(t_1,\dots,t_N\in T\) satisfy \(d(t_i,t_j)\geq a\) for \(i\ne j\), then there is a universal constant \(c>0\) such that
\[
\E\sup_{t\in T}X_t\geq ca\sqrt{\log N}.
\]
\end{lemma}

We use Sudakov's minoration as a standard result on Gaussian processes; see
\cite{Sudakov1969} for the original result, and
\cite[Theorem 7.5.1]{Vershynin}, \cite{LedouxTalagrand}, or
\cite{TalagrandUpperLower} for modern treatments.
We are now ready to prove the lower estimate.

\begin{proposition}
There exists \(c_m>0\) such that
\[
\Prob\left(M(Q_n)\geq c_m n^{(m+1)/2}\right)\to1.
\]
\end{proposition}

\begin{proof}
Let \(X_\eps=Q_n(\eps)\). Take \(\mathcal A_n\subset\{-1,1\}^n\) as in Lemma 3. By Lemma 4, distinct points of \(\mathcal A_n\) are separated by at least \(b_{m,\delta}n^{m/2}\) in the canonical metric. Sudakov's minoration gives
\[
\E\max_{\eps\in\mathcal A_n}X_\eps
\geq c b_{m,\delta}n^{m/2}\sqrt{\log|\mathcal A_n|}
\geq c_m n^{(m+1)/2}.
\]
Since
\[
M(Q_n)=\max_{\eps\in\{-1,1\}^n}|X_\eps|,
\]
we get
\[
\E M(Q_n)\geq c_m n^{(m+1)/2}.
\]

It remains to pass from expectation to high probability. Let
\[
F((g_S)_{|S|=m})
=
V(Q_n)
=
M(Q_n).
\]
From the estimate above, there exists a constant \(A_m>0\), independent of
\(n\), such that
\[
\E F\geq A_m n^{(m+1)/2}.
\]

We claim that \(F\) is Lipschitz with constant
\[
L_n\leq \left(m!\binom{n}{m}\right)^{1/2}\leq n^{m/2}.
\]
Indeed, if \(g=(g_S)_{|S|=m}\) and \(g'=(g'_S)_{|S|=m}\), then, for every
\(\eps\in\{-1,1\}^n\),
\[
|Q_n^g(\eps)-Q_n^{g'}(\eps)|
\leq
\sqrt{m!}
\left(\sum_{|S|=m}(g_S-g'_S)^2\right)^{1/2}
\left(\sum_{|S|=m}\eps_S^2\right)^{1/2}.
\]
Since \(\eps_S^2=1\), we get
\[
|Q_n^g(\eps)-Q_n^{g'}(\eps)|
\leq
\left(m!\binom{n}{m}\right)^{1/2}\|g-g'\|_2.
\]
Taking the maximum over \(\eps\) proves the claimed Lipschitz estimate.

By Gaussian concentration for Lipschitz functions
\cite{LedouxTalagrand,Vershynin},
\[
\Prob(F\leq \E F-u)
\leq
\exp\left(-\frac{u^2}{2L_n^2}\right).
\]
Taking \(u=\E F/2\), we obtain
\[
\Prob\left(F\leq \frac12\E F\right)
\leq
\exp\left(-\frac{(\E F)^2}{8L_n^2}\right).
\]
Using
\[
\E F\geq A_m n^{(m+1)/2},
\qquad
L_n\leq n^{m/2},
\]
we get
\[
\frac{(\E F)^2}{8L_n^2}
\geq
\frac{A_m^2 n^{m+1}}{8n^m}
=
\frac{A_m^2}{8}n.
\]
Therefore
\[
\Prob\left(F\leq \frac12\E F\right)
\leq
\exp\left(-\frac{A_m^2}{8}n\right).
\]
In particular, with probability tending to one,
\[
F\geq \frac12\E F
\geq
\frac{A_m}{2}n^{(m+1)/2}.
\]
Setting \(c_m=A_m/2\) proves the assertion.
\end{proof}

\section{The repeated-variable part}

We now estimate the part of \(P_n\) formed by the monomials with repeated
variables. We decompose
\[
R_n=\sum_{r=1}^{m-1}R_{n,r},
\]
where \(R_{n,r}\) contains those monomials of total degree \(m\) involving
exactly \(r\) distinct variables. Thus \(r\leq m-1\), since at least one variable
must be repeated. The following proposition controls repeated-variable blocks.

\begin{proposition}
For each \(1\leq r\leq m-1\),
\[
M(R_{n,r})=O_{\Prob}(n^{(r+1)/2}).
\]
Consequently,
\[
M(R_n)=O_{\Prob}(n^{m/2}).
\]
\end{proposition}

\begin{proof}
Fix \(1\leq r\leq m-1\). Let \(\mathcal C_{m,r}\) be the set of ordered compositions of \(m\) into
\(r\) positive parts, namely
\[
\mathcal C_{m,r}
=
\left\{
a=(a_1,\dots,a_r)\in\mathbb N^r:
a_1+\cdots+a_r=m
\right\}.
\]
Thus
\[
|\mathcal C_{m,r}|=\binom{m-1}{r-1},
\]
and in particular this number depends only on \(m\), not on \(n\).

Using the convention
\[
1\leq i_1<\cdots<i_r\leq n,
\]
we decompose
\[
R_{n,r}
=
\sum_{a\in\mathcal C_{m,r}}R_{n,r,a},
\]
where
\[
R_{n,r,a}(x)=
\sum_{1\leq i_1<\cdots<i_r\leq n}
c_a g^{(a)}_{i_1,\dots,i_r}
x_{i_1}^{a_1}\cdots x_{i_r}^{a_r},
\]
and
\[
c_a=
\sqrt{\frac{m!}{a_1!\cdots a_r!}}.
\]
Here \(g^{(a)}_{i_1,\dots,i_r}\) denotes the Gaussian coefficient associated
with the multiindex \(\alpha\) satisfying
\[
\alpha_{i_j}=a_j\quad (1\leq j\leq r),
\qquad
\alpha_i=0\quad\text{otherwise}.
\]
For each fixed \(r\) and \(a\), the coefficients
\[
g^{(a)}_{i_1,\dots,i_r}
\]
are independent standard Gaussian variables.

It is enough to prove the required upper bound for each fixed block
\(R_{n,r,a}\). Indeed, note that the number of ``patterns'' \(a\in\mathcal C_{m,r}\) depends
only on \(m\).

Define the multilinear form on \(r\) copies of the cube by
\[
\widetilde R_{n,r,a}(y^{(1)},\dots,y^{(r)})
=
\sum_{1\leq i_1<\cdots<i_r\leq n}
c_a g^{(a)}_{i_1,\dots,i_r}
y^{(1)}_{i_1}\cdots y^{(r)}_{i_r}.
\]
Set
\[
M_r(\widetilde R_{n,r,a})
=
\max_{y^{(1)},\dots,y^{(r)}\in[-1,1]^n}
\left|
\widetilde R_{n,r,a}(y^{(1)},\dots,y^{(r)})
\right|.
\]

If \(x\in[-1,1]^n\), put
\[
y_i^{(j)}=x_i^{a_j}
\qquad
(1\leq i\leq n,\ 1\leq j\leq r).
\]
Then \(y^{(j)}\in[-1,1]^n\) for every \(j\), and
\[
R_{n,r,a}(x)
=
\widetilde R_{n,r,a}(y^{(1)},\dots,y^{(r)}).
\]
Therefore
\[
M(R_{n,r,a})
\leq
M_r(\widetilde R_{n,r,a}).
\]

Since \(\widetilde R_{n,r,a}\) is affine in each scalar coordinate of each
vector variable separately, its maximum over \(([-1,1]^n)^r\) is attained at
\(r\)-tuples of vertices. Hence
\[
M_r(\widetilde R_{n,r,a})
=
\max_{\eps^{(1)},\dots,\eps^{(r)}\in\{-1,1\}^n}
\left|
\widetilde R_{n,r,a}(\eps^{(1)},\dots,\eps^{(r)})
\right|.
\]

For fixed signs \(\eps^{(1)},\dots,\eps^{(r)}\), the random variable
\[
\widetilde R_{n,r,a}(\eps^{(1)},\dots,\eps^{(r)})
\]
is centered Gaussian. Since all signs have modulus one and the coefficients are
independent, its variance is
\[
c_a^2\binom{n}{r}.
\]
In particular,
\[
\Var\left(
\widetilde R_{n,r,a}(\eps^{(1)},\dots,\eps^{(r)})
\right)
\leq
c_a^2 n^r.
\]
Therefore the Gaussian tail estimate gives
\[
\Prob\left(
\left|
\widetilde R_{n,r,a}(\eps^{(1)},\dots,\eps^{(r)})
\right|>t
\right)
\leq
2\exp\left(-\frac{t^2}{2c_a^2 n^r}\right).
\]

There are \(2^{rn}\) choices of
\[
(\eps^{(1)},\dots,\eps^{(r)})\in(\{-1,1\}^n)^r.
\]
By the union bound,
\[
\Prob\left(
M_r(\widetilde R_{n,r,a})>t
\right)
\leq
2^{rn+1}
\exp\left(-\frac{t^2}{2c_a^2 n^r}\right).
\]

Now take
\[
t=Cn^{(r+1)/2}.
\]
Then
\[
\Prob\left(
M_r(\widetilde R_{n,r,a})>Cn^{(r+1)/2}
\right)
\leq
2\exp\left(
rn\log2-\frac{C^2}{2c_a^2}n
\right).
\]
Since
\[
c_a^2=\frac{m!}{a_1!\cdots a_r!}\leq m!,
\]
we may choose \(C>0\), depending only on \(m\), such that
\[
\frac{C^2}{2m!}>(m-1)\log2.
\]
Then, for every \(1\leq r\leq m-1\) and every
\(a\in\mathcal C_{m,r}\),
\[
\frac{C^2}{2c_a^2}
\geq
\frac{C^2}{2m!}
>
(m-1)\log2
\geq
r\log2.
\]
Therefore
\[
\Prob\left(
M_r(\widetilde R_{n,r,a})>Cn^{(r+1)/2}
\right)
\longrightarrow0.
\]
Hence
\[
M_r(\widetilde R_{n,r,a})
=
O_{\Prob}(n^{(r+1)/2}),
\]
with constants depending only on \(m\).

Using the deterministic comparison
\[
M(R_{n,r,a})
\leq
M_r(\widetilde R_{n,r,a}),
\]
we get the same upper bound for the original block:
\[
M(R_{n,r,a})
=
O_{\Prob}(n^{(r+1)/2}).
\]

Finally, since
\[
R_{n,r}
=
\sum_{a\in\mathcal C_{m,r}}R_{n,r,a}
\]
and \(|\mathcal C_{m,r}|\) depends only on \(m\), the finite sum of these
bounds gives
\[
M(R_{n,r})
=
O_{\Prob}(n^{(r+1)/2}).
\]
Since
\[
R_n=\sum_{r=1}^{m-1}R_{n,r},
\]
and the largest exponent \((r+1)/2\) occurs when \(r=m-1\), we obtain
\[
M(R_n)
=
O_{\Prob}(n^{m/2}).
\]
\end{proof}

\section{Proof of the main result}

We first prove the quantitative estimate from which the almost norming
statement follows.

\begin{theorem}
Let \(m\geq1\) be fixed. For each \(n\), let \(P_n\) be chosen according to
the Bombieri Gaussian measure \(\gamma_{m,n}\) on \(\cP_m(\R^n)\). Then
\[
1-\frac{V(P_n)}{M(P_n)}
=
O_{\Prob}(n^{-1/2}).
\]
\end{theorem}

\begin{proof}
The case \(m=1\) is immediate, since linear functionals attain their norm on
the cube at vertices. Thus assume \(m\geq2\).

The estimates above give
\[
M(R_n)=O_{\Prob}(n^{m/2})
\]
and, with high probability,
\[
M(Q_n)\geq c_m n^{(m+1)/2}.
\]
Therefore
\[
t_n=\frac{M(R_n)}{M(Q_n)}
=
O_{\Prob}(n^{-1/2}).
\]
By the deterministic reduction,
\[
1-\frac{V(P_n)}{M(P_n)}
\leq
2t_n.
\]
Hence
\[
1-\frac{V(P_n)}{M(P_n)}
=
O_{\Prob}(n^{-1/2}).
\]
\end{proof}

As a consequence, the vertices are asymptotically norming in the following
quantitative sense.

\begin{corollary}
Let \(m\geq1\) be fixed. For every \(0<\beta<1/2\),
\[
\gamma_{m,n}\left\{
P\in\cP_m(\R^n):
V(P)\geq (1-n^{-\beta})M(P)
\right\}
\longrightarrow1.
\]
In fact, if \(P_n\) is chosen according to \(\gamma_{m,n}\), then
\[
n^\beta\left(
1-\frac{V(P_n)}{M(P_n)}
\right)
\longrightarrow0
\]
in probability. In particular,
\[
\frac{V(P_n)}{M(P_n)}
\longrightarrow1
\]
in probability.
\end{corollary}

\begin{proof}
By the theorem,
\[
1-\frac{V(P_n)}{M(P_n)}
=
O_{\Prob}(n^{-1/2}).
\]
Multiplying by \(n^\beta\), with \(0<\beta<1/2\), gives
\[
n^\beta\left(
1-\frac{V(P_n)}{M(P_n)}
\right)
=
O_{\Prob}(n^{\beta-1/2}).
\]
Since \(n^{\beta-1/2}\to0\), the last expression converges to \(0\) in
probability. This proves the stronger statement, and hence the stated
measure convergence.
\end{proof}

\medskip
\noindent\textbf{Use of generative AI.} During the preparation of this manuscript, OpenAI's ChatGPT was used as an assistive tool for language editing, manuscript organization, bibliographic searches, and the discussion of mathematical arguments. All mathematical content, proofs, calculations, and references included in the final manuscript were independently reviewed and verified by the authors. The authors assume full responsibility for the content of this work.


\begin{thebibliography}{99}

\bibitem{BeauzamyBombieriEnfloMontgomery1990}
B. Beauzamy, E. Bombieri, P. Enflo and H. L. Montgomery,
\emph{Products of polynomials in many variables},
J. Number Theory \textbf{36} (1990), no. 2, 219--245.

\bibitem{CarandoDimantPinasco2011}
D. Carando, V. Dimant and D. Pinasco,
\emph{On the convergence of random polynomials and multilinear forms},
J. Funct. Anal. \textbf{261} (2011), no. 8, 2135--2163.

\bibitem{CarandoZalduendo2003}
D. Carando and I. Zalduendo,
\emph{Where do homogeneous polynomials attain their norm?},
Publ. Math. Debrecen \textbf{62} (2003), no. 1--2, 19--28.

\bibitem{EdelmanKostlan1995}
A. Edelman and E. Kostlan,
\emph{How many zeros of a random polynomial are real?},
Bull. Amer. Math. Soc. \textbf{32} (1995), no. 1, 1--37.

\bibitem{Kostlan1993}
E. Kostlan,
\emph{On the distribution of roots of random polynomials},
in \emph{From Topology to Computation: Proceedings of the Smalefest}, Springer,
1993, 419--431.

\bibitem{Latala2006}
R. Lata{\l}a,
\emph{Estimates of moments and tails of Gaussian chaoses},
Ann. Probab. \textbf{34} (2006), no. 6, 2315--2331.

\bibitem{LedouxTalagrand}
M. Ledoux and M. Talagrand,
\emph{Probability in Banach Spaces: Isoperimetry and Processes},
Ergebnisse der Mathematik und ihrer Grenzgebiete, Springer, 1991.

\bibitem{PerezGarciaVillanueva2004}
D. P\'erez-Garc\'ia and I. Villanueva,
\emph{Where do homogeneous polynomials on \(\ell_1^n\) attain their norm?},
J. Approx. Theory \textbf{127} (2004), no. 1, 124--133.

\bibitem{PinascoSmuclerZalduendo2021}
D. Pinasco, E. Smucler and I. Zalduendo,
\emph{Orthant probabilities and the attainment of maxima on a vertex of a simplex},
Linear Algebra Appl. \textbf{610} (2021), 785--803.

\bibitem{PinascoZalduendo2012}
D. Pinasco and I. Zalduendo,
\emph{A probabilistic approach to polynomial inequalities},
Israel J. Math. \textbf{190} (2012), 67--82.

\bibitem{PinascoZalduendo2019}
D. Pinasco and I. Zalduendo,
\emph{On the measure of polynomials attaining maxima on a vertex},
Math. Inequal. Appl. \textbf{22} (2019), no. 2, 421--432.

\bibitem{Sudakov1969}
V. N. Sudakov,
\emph{Gaussian measures, Cauchy measures and \(\varepsilon\)-entropy},
Soviet Math. Dokl. \textbf{10} (1969), 310--313.

\bibitem{TalagrandUpperLower}
M. Talagrand,
\emph{Upper and Lower Bounds for Stochastic Processes: Modern Methods and Classical Problems},
Ergebnisse der Mathematik und ihrer Grenzgebiete, Springer, 2014.

\bibitem{Vershynin}
R. Vershynin,
\emph{High-Dimensional Probability: An Introduction with Applications in Data Science},
Cambridge Series in Statistical and Probabilistic Mathematics, Cambridge
University Press, 2018.

\end{thebibliography}
\end{document}